\documentclass[12pt]{amsart}

\usepackage{amsmath, amsfonts,amsthm,times,graphics}

\usepackage[active]{srcltx}

 \makeatletter
\renewcommand*\subjclass[2][2010]{%
  \def\@subjclass{#2}%
  \@ifundefined{subjclassname@#1}{%
    \ClassWarning{\@classname}{Unknown edition (#1) of Mathematics
      Subject Classification; using '2010'.}%
  }{%
    \@xp\let\@xp\subjclassname\csname subjclassname@#1\endcsname
  }%
}

 \makeatother
\newtheorem{theorem}{Theorem}[section]
\newtheorem{lemma}[theorem]{Lemma}
\newtheorem{corollary}[theorem]{Corollary}
\newtheorem{claim}[theorem]{Claim}
\newtheorem{proposition}[theorem]{Proposition}
\newtheorem{conjecture}[theorem]{Conjecture}
\newtheorem{question}[theorem]{Question}
\newtheorem{identity}[theorem]{Identity}
\theoremstyle{definition}
\newtheorem{definition}[theorem]{Definition}

 \newtheorem{example}[theorem]{Example}

\newtheorem{remark}[theorem]{Remark}

 \makeatletter
\renewcommand*\subjclass[2][2010]{%
  \def\@subjclass{#2}%
  \@ifundefined{subjclassname@#1}{%
    \ClassWarning{\@classname}{Unknown edition (#1) of Mathematics
      Subject Classification; using '1991'.}%
  }{%
    \@xp\let\@xp\subjclassname\csname subjclassname@#1\endcsname
  }%
}

 \makeatother

\begin{document}
\title[On some discrete random variables...]{On some discrete random 
variables arising
from recent study on statistical analysis of compressive sensing}

\author{Romeo Me\v strovi\' c}
\address{Maritime Faculty Kotor, University of Montenegro, 
85330 Kotor, Montenegro} 
\email{romeo@ac.me}

 \subjclass{94A12, 60C05,  11A07, 05A10}
\keywords{Compressive sensing, Missing samples, Discrete Fourier transform,
Complex-valued discrete random variable,   
$k$th moment, Combinatorial identity, Eisenstein's irreducibility 
criterion, Welch bound.}
 
\begin{abstract}
  
The recent paper \cite{ssa} provides a statistical analysis for efficient 
detection of signal components when missing data samples are present.
Here we focus our attention to some complex-valued discrete random 
variables  $X_l(m,N)$ ($0\le l\le N-1$, 
$1\le M\le N$), which are closely related 
to the random variables
investigated by LJ. Stankovi\'c, S. Stankovi\'c and M. Amin in \cite{ssa}. 
In particular, by using a combinatorial approach, 
we prove that for $l\not=0$ the expected value of $X_l(m,N)$ is equal to zero, 
and we  deduce the expression for the variance
of the random variables $X_l(m,N)$. The same results are also deduced 
for the real part $U_l(m,N)$ and the imaginary part $V_l(m,N)$ of $X_l(m,N)$,  
as well as the facts that the $k$th moments
of $U_l(m,N)$ and $V_l(m,N)$ are equal to zero for every 
positive integer $k$ which  is not divisible by $N/\gcd(N,l)$.  
Moreover, some  additional assertions and  examples  concerning 
the random variables $X_l(m,N)$, $U_l(m,N)$  and $V_l(m,N)$ 
are also presented.
   \end{abstract}  
  \maketitle

\section{Motivation, definitions and related examples}

Recently,  LJ. Stankovi\'c, S. Stankovi\'c and M. Amin \cite{ssa}  
 provided a statistical analysis for efficient 
detection of signal components when missing data samples are present. 
As noticed in \cite{ssa}, this analysis is important for both the area
of L-statistics and {\it compressive sensing}.
In both cases, few samples are available due to either noisy sample 
elimination of random undersampling signal strategies. For more 
information on the development of compressive sensing (also known as {\it compressed sensing}, 
{\it compressive sampling}, or {\it sparse recovery}), see \cite{do}, \cite{fr}, 
\cite[Chapter 10]{s1} and \cite{sdt}. For an excellent survey on this topic 
with applications and related references, see  \cite{sr} (also see \cite{op}).

In \cite[Section 2]{ssa} (cf. \cite[Section II]{sso} and \cite{sid}) 
\cite[Section 2]{sid}), the authors considered a set of $N$ signal values $\Theta$ given by
   $$
\Theta=\{ s(1), s(2), \ldots, s(N)\},
  $$
 where
a signal which is sparse in the Discrete Fourier Transform (DFT) domain can 
be written as 
   $$
s(n)=\sum_{i=1}^KA_ie^{j2\pi k_{0i}n/N},\quad n\in\{1,2,\ldots,N\}, \leqno(1)
  $$ 
and  the level of sparsity is $K\ll N$, while $A_i$ and $k_{0i}$
denote amplitudes and frequencies of the signal components, respectively.
Notice that the relation $K\ll N$  means that most of components
of a considered signal are zero. 
The  application  of the DFT to the above sequence $\Theta$  leads to the set 
$\Phi(l,N)$ of the form (the set $\Phi$ in the equality (3) of \cite{ssa}): 
    $$
\Phi(l,N)=\{e^{-j2nl\pi/N}: n=1,2,\ldots,N\} \,\, {\rm with\,\, some\,\,
fixed}\,\,l\in \{0,1,\ldots, N-1\}.\leqno(2)
   $$

As usually, throughout our considerations we use the 
term ``multiset'' (often written as ``set'') to mean ``a totality having possible 
multiplicities''; so that two (multi)sets will be counted as equal if 
and only if they have the same elements with identical multiplicities.

Notice that (2) for $l=0$ imlies that 
 $$
\Phi(0,N)=\{\underbrace{1,1,\ldots,1}_N\}.
$$
Moreover, it is obvious that $\Phi(l,N)$ given by (2) is a set consisting of  
$N$ (distinct) elements if and only if  $l$ and $N$ are relatively prime 
positive integers.

Let ${\mathcal M}$ denote the collection of all
multisets $\Phi(l,N)$ of the form (2), i.e., 
   $$
{\mathcal M}=\{\Phi(l,N): N=1,2,\ldots;\, l=0,1,\ldots,N-1\}.
   $$ 
For simplicity and for our computational purposes, for   fixed $N\ge 1$
and $l$  such that $1\le l\le N-1$, in the sequel we shall often write 
$w:=e^{-j2l\pi/N}$. Accordingly, for each $l=1,2,\ldots, N-1$ the multiset 
$\Phi(l,N)$ defined by (2) can be written as    
  $$
\Phi(l,N)=\{w,w^2,\ldots,w^N\}.\leqno(3)
  $$
Furthermore (see \cite[Eq. (3)]{ssa}),  we have  
  $$
w+w^2+\cdots +w^{N}=0,\leqno(4)
   $$
or if we take $x(n)=e^{-j2nl\pi/N}$ ($n=1,2,\ldots N$), it is equivalent to 
   $$
x(1)+x(2)+\cdots +x(N)=0.
   $$
Here, as always in the sequel, we will assume that the signal length $N$
is an arbitrary fixed positive integer.  
Accordingly, assuming that $K=1$ and $A_1=1$, 
for any fixed $l\in\{1,2,\ldots,N-1\}$, 
in \cite{ssa} the authors considered a subset
$\Psi(l,N;m)$ of $\Phi(l,N)$ consisting of $m\ll N$ randomly positioned 
available samples (measurements), i.e., 
     $$
\Psi(l,N;m)=\{y(1), y(2),\ldots, y(m)\}\subset \Phi(l,N).\leqno(5)
      $$

Then the random variable  corresponding  
to the  DFT  over the available set of samples from $\Phi(l,N)$
is given by 
    $$
X_l(m,N)=:X_l(m)=\sum_{n=1}^my(n)=\sum_{n=1}^N(x(n)+\varepsilon(n)),\leqno(6)
   $$
where
     $$
(7)\,  \varepsilon(n)=\left\{ 
 \begin{array}{ll}
0 &  \textrm{for  remaining signal samples}\\
 -x(n)   =-\exp(-2jl\pi/N) & \textrm{for removed (unavailable) 
signal samples}.\\
\end{array}\right. 
  $$
Observe that $X_l(m,N)$ defined by (6) is a complex-valued discrete random 
variable formed as a sum of $m$ randomly positioned samples 
$y(1), y(2),\ldots,y(m)\in \Psi(l,N;m)$ 

\noindent $\subseteq \Phi(l,N)$.
 Let us notice that the theory currently available 
on compressive sensing  predicts that sampling 
sets chosen uniformly at random among all possible sets of 
a given fixed cardinality work well (see, e.g., \cite[Chapter 12]{fr}).
For some variations of this random variable see \cite{ssd}. 
 If the number $m$ of randomly positioned available samples (measurements)
is not fixed, but randomly chosen (i.e., if the number of terms 
in the sum (6) is itself a random variable), then the related
random variable $X_l(m,N)$ can be replaced (generalized) with the
corresponding the so-called  {\it compound random variable}.
These random variables were firstly systematicaly  studied   by W. Feller 
in his famous book \cite{f1}.
A combinatorial approach to the introductory study of the  
compound random variable followed by several examples was 
given in  \cite{m}.    

Notice that in the above definition of the random variable 
$X_l(m,N)$ given by (6), the number of randomly positioned samples,
$m$, is a fixed positive integer such that $1\le m\le N$.
We believe that in probabilistic study of sparse signal recovery 
it can be of interest the complex-valued discrete random variable
$\widetilde{X}_l(m,N)$  which  
may be considered as a   random analogue
(or ``free companion'' random variable)
 of the random 
variable $X_l(m,N)$, and it is studied and defined  in \cite{m4} 
 as follows.
   \begin{definition} 
Let $N$, $l$ and $m$ be nonnegative  integers such that $0\le l\le N-1$
and $1\le m\le N$.  Let  $B_n$ $(n=1,\ldots,N)$ be a sequence 
of independent identically distributed  {\it Bernoulli 
random variables} (binomial distributions) taking only
the values 0 and 1 with probability 0 and $m/N$, respectively,
i.e.,
   $$
B_n=\left\{ 
 \begin{array}{ll}
0 & \mathrm{with\,\,probability \,\,} 1-\frac{m}{N}\\
1 & \mathrm{with\,\,probability \,\,} \frac{m}{N}. 
 \end{array}\right.\leqno(8)
   $$
Then the discrete random variable $\widetilde{X}_l(m,N)$ is defined as
a sum
  $$
\widetilde{X}_l(m,N)=\sum_{n=1}^N\exp\left({-\frac{2jn l\pi}{N}}
\right)B_n.\leqno(9)
   $$
\end{definition}

From Definition 1.1 we see that the range of the random variable 
$\widetilde{X}_l(m,N)$ consists of all possible $2^N-1$ sums of the elements 
of (multi)set $\{e^{-j2nl\pi/N}:\, n=1,2,\ldots,N\}$.    

Observe that for $l=0$  $\widetilde{X}_l(m,N)$ becomes
   $$
\widetilde{X}_0(m,N)=\sum_{k=1}^NB_k\sim B\left(N,\frac{m}{M}\right),
\leqno(10)
   $$ 
where $B\left(N,m/N\right)$ is the  binomial distribution
with  parameters $N$ and $p=m/N$ and the probability mass function 
given by 
  $$
{\rm Prob}\left(B\left(N,\frac{m}{N}\right)=k\right)={N\choose k}
\left(\frac{m}{N}\right)^k\left(1-\frac{m}{N}\right)^{N-k}, 
\,\, k=0,1,\ldots , N.
   $$
Notice also that a {\it Bernoulli probability model}, similar 
to the distribution $\widetilde{X}_l(m,N)$ defined by (9), was often used in 
the famous paper \cite{crt} by   Cand\`{e}s, Romberg and Tao. Moreover, 
the random variables  $\widetilde{X}_l(m,N)$ have some similar probabilistic 
characteristics to those of $X_l(m,N)$. 

Now we return to the random variable $X_l(m)=X_l(m,N)$
defined by (6).
Let $(\xi_1, \xi_2, \ldots, \xi_N)$ be a $n$-tuple of integers 
$\xi_n$ which are  chosen uniformly at random from the set $\{0,1\}$ 
under the condition that 
   $$
\xi_1+\xi_2+\cdots+\xi_N=m.
   $$
Then the discrete random variable defined by (6) and considered
 in \cite[p. 402]{ssa} can be written as a sum
  $$
X_l(m)=\sum_{n=1}^N\xi_n\exp\left({-\frac{2jn l\pi}{N}}\right).
\leqno(11)
  $$
In fact, the above representation   means that the sparse signal 
considered in \cite{ssa} ``comes'' from the set of values of the  random variable 
$X_l(m)$.
In view of the above considerations,   in the form of  
its distribution law,  $X_l(m)$ may be defined as follows.

% Def . 1.2
  \begin{definition}
Let $N$, $l$ and $m$  be arbitrary nonnegative integers 
such that $0\le l\le N-1$ and $1\le m\le N$.
Let $\Phi(l,N) \in {\mathcal M}$ be a multiset defined as 
   $$
\Phi(l,N)=\{e^{-j2nl\pi/N}:\, n=1,2,\ldots,N\}.\leqno(12)
   $$
Define the discrete complex-valued random variable $X_l(m,N)=X_l(m)$ as
  \begin{eqnarray*}
 && \mathrm{Prob}\left(X_l(m)
 = \sum_{i=1}^me^{-j2n_il\pi/N}\right)\\
(13) &= &\frac{1}{{N\choose m}}\cdot \big|\{\{t_1,t_2,\ldots,t_m\}
\subset\{1,2,\ldots,N\}: 
\sum_{i=1}^me^{-j2t_il\pi/N}=\sum_{i=1}^me^{-j2n_il\pi/N} \big|
\quad\qquad\\
& =&:\frac{q(n_1,n_2,\ldots, n_m)}{{N\choose m}},
  \end{eqnarray*}
  where $\{n_1,n_2,\ldots, n_m\}$ is an arbitrary fixed
 subset of $\{1,2,\ldots,N\}$ such that $1\le n_1<n_2<\cdots <n_m\le N$;
moreover, $q(n_1,n_2,\ldots, n_m)$ is the cardinality of a collection
 of all subsets
$\{t_1,t_2,\ldots,t_m\}$ of the set $\{1,2,\ldots,N\}$ such that 
$\sum_{i=1}^me^{-j2t_il\pi/N}=\sum_{i=1}^me^{-j2n_il\pi/N}$.   
  \end{definition}

% Rem 1.3
   \begin{remark}
The above definition is correct  in view of the fact that there are 
${N\choose m}$ index sets $T\subset \{1,2,\ldots,N \}$ with 
$m$ elements. Notice also that this quantity grows (in some sense)
exponentially with $m$ and $N$. For a sake  of understanding this definition,
see Examples 1.4 and 1.5 given below. Moreover, a very short, but not 
strongly exact version of Definition 1.2 is given as follows.
   \end{remark}
\vspace{1mm}

% Def . 1.2'
  \noindent {\bf Definition 1.2'.}
Let $N$, $l$ and $m$  be arbitrary nonnegative integers 
such that $0\le l\le N-1$ and $1\le m\le N$.
Let $\Phi(l,N) \in {\mathcal M}$ be a multiset defined as 
   $$
\Phi(l,N)=\{e^{-j2nl\pi/N}:\, n=1,2,\ldots,N\}.
   $$
Choose a random subset  $S$ of size $m$ (the so-called $m$-element 
subset) without replacement from the set 
$\{1,2,\ldots, N\}$. Then the 
 complex-valued discrete random variable $X_l(m,N)=X_l(m)$ is defined as a sum
  $$
X_l(m)=\sum_{n\in S}e^{-j2nl\pi/N}.
  $$

\vspace{1mm}

% EX 1.5
\begin{example} Consider the multiset 
 $$
\Phi(2,6)=\{e^{-j2n\pi/3}:\, n=1,2,3,4,5,6\}.
  $$
If for brevity we put $\varepsilon = e^{-2j\pi/3}=(-1-j\sqrt{3})/2$, then 
obviously $\Phi(2,6)$ can be written as
 $$
\Phi(2,6)=\{\varepsilon,\varepsilon,\varepsilon^2,\varepsilon^2,1,1\}.
  $$
Then accordingly to Definition 1.2,  $X_1(1)$
   is the uniform  random variable with
   $$ 
\mathrm{Prob}\left(X_1(1)=\varepsilon\right)=
\mathrm{Prob}\left(X_1(1)=\varepsilon^2\right)=
\mathrm{Prob}\left(X_1(1)=1\right)=\frac{2}{6}=\frac{1}{3}.
  $$
If we put $X_1(1)=U+jV$, where $U$ is the {\it real part} 
 and $V$ is the {\it imaginary part} of $X_1(1)$, then since 
$\varepsilon^2 = e^{-4j\pi/3}=(-1+j\sqrt{3})/2$,
 a routine calculation 
gives the following probability laws of $U$ and $V$:
  $$ 
\mathrm{Prob}(U=1)=\frac{1}{3},  
\mathrm{Prob}\left(U = -\frac{1}{2}\right)=\frac{2}{3}; 
  $$
  $$ 
\mathrm{Prob}(V=0)=\mathrm{Prob}\left(V =\frac{\sqrt{3}}{2}\right)  =
 \mathrm{Prob}\left(V =-\frac{\sqrt{3}}{2}\right)=\frac{1}{3}. 
  $$
From the above two distribution laws, we immediately obtain the following 
   probability laws of $U^2$, $V^2$ and $UV$:
  $$ 
\mathrm{Prob}(U^2=1)=\frac{1}{3},  
\mathrm{Prob}\left(U^2 = \frac{1}{4}\right)=\frac{2}{3},
  $$
  $$ 
\mathrm{Prob}(V^2=0)=\frac{1}{3},  
\mathrm{Prob}\left(V^2 = \frac{3}{4}\right)=\frac{2}{3},
  $$
and 
  $$ 
\mathrm{Prob}(UV=0)=\frac{1}{3},  
\mathrm{Prob}\left(UV = \frac{\sqrt{3}}{2}\right)=\frac{1}{9},
  \mathrm{Prob}\left(UV = -\frac{\sqrt{3}}{2}\right)=\frac{1}{9},
   $$
  $$
\mathrm{Prob}\left(UV = \frac{\sqrt{3}}{4}\right)=\frac{2}{9},
  \mathrm{Prob}\left(UV = -\frac{\sqrt{3}}{4}\right)=\frac{2}{9}.
  $$
 
Generally, if  $X=U+jV$ is a complex-valued random variable, 
then the expected value of its square is defined as  
   $$
\Bbb E[X^2]=\Bbb E[U^2]+\Bbb E[V^2]-2j\Bbb E[UV].\leqno(14)
  $$
This expression together with above derived probability law implies that
   $$
\Bbb E[(X_1(1))^2]=\left(\frac 13+\frac 16\right)+\frac 12-2j\cdot 0=1.
   $$
 \end{example}

% EX 1.5
\begin{example} Consider the set 
 $$
\Phi(1,4)=\{e^{-jn\pi/2}:\, n=1,2,3,4\}.
   $$
Since $e^{-j\pi/2}=-j$, we have
    $$ 
\Phi(1,4)=\{1,-1,j,-j\}.\leqno(15)
  $$
Then accordingly to Definition 1.2, the probability law of  $X_1(2)$
   is given by  
   \begin{equation*}\begin{split}
\mathrm{Prob}(X_1(2)=0)&=\frac{1}{3},   
\mathrm{Prob}(X_1(2)=1+j)=\mathrm{Prob}(X_1(2)=-1+j)\\
&=\mathrm{Prob}(X_1(2)=-1-j)=\mathrm{Prob}(X_1(2)=1-j)=\frac{1}{6}.
  \end{split}\end{equation*}
If we set $X_1(2)=U+jV$, where $U$ is the real part 
 and $V$ is the imaginary part of $X_1(2)$, then  a simple calculation 
implies that both random variables $U$ and $V$ are uniformly 
distributed, i.e.,
  $$ 
\mathrm{Prob}(U=0)=\mathrm{Prob}(U=1)=\mathrm{Prob}(U=-1)=\frac{1}{3},
  $$
$$ 
\mathrm{Prob}(V=0)=\mathrm{Prob}(V=1)=\mathrm{Prob}(V=-1)=\frac{1}{3}.
  $$
The random variable $(X_1(2))^2$ is also uniformly distributed; namely, 
  $$
\mathrm{Prob}((X_1(2))^2=0)=\mathrm{Prob}((X_1(2))^2=2j)=
\mathrm{Prob}((X_1(2))^2=-2j)=\frac{1}{3}.
  $$
Moreover, the distribution laws of  $(X_1(2))^3$ and
$(X_1(2))^4$ are respectively given as follows:
    \begin{equation*}\begin{split}
\mathrm{Prob}((X_1(2))^3=0)&=\frac{1}{3},   
\mathrm{Prob}((X_1(2))^3=2+2j)=\mathrm{Prob}(X_1(2)=-2+2j)\\
&=\mathrm{Prob}((X_1(2))^3=-2-2j)=\mathrm{Prob}(X_1(2)=2-2j)=\frac{1}{6},
  \end{split}\end{equation*}
  $$
\mathrm{Prob}((X_1(2))^4=0)= \frac{1}{3},\mathrm{Prob}((X_1(2))^4=4)=
\frac{2}{3}.
  $$
Notice that from the above described distributions it follows that
    $$
\Bbb E [X_1(2)]=\Bbb E [(X_1(2))^2]=\Bbb E [(X_1(2))^3]=0
\quad {\rm and}\quad \Bbb E [(X_1(2))^4]=\frac{8}{3}.\leqno(16)
    $$
Moreover, since  
$$
\mathrm{Prob}(|X_1(2)|=0)= 1/3, \mathrm{Prob}(|X_1(2)|=\sqrt{2})= 2/3, 
\mathrm{Prob}(|X_1(2)|^2=0)= 1/3
  $$
and $\mathrm{Prob}(|X_1(2)|=2)= 2/3$, by definition, we obtain that 
the variance of $X_1(2)$ is equal to 
   $$
{\rm Var}[X_1(2)]=\Bbb E[|X_1(2)|^2] -|\Bbb E[X_1(2)]|^2=
\frac{4}{3}-0=\frac{4}{3}.
  $$ 

Furthermore,  using  (15), a routine calculation shows that $X_1(3)$ is the 
uniformly distributed random variable, i.e.,
    \begin{eqnarray*} 
\mathrm{Prob}(X_1(3)=1)&=&\mathrm{Prob}(X_1(3)=-1)=
\mathrm{Prob}(X_1(3)=j)\\
&=&\mathrm{Prob}(X_1(3)=-j)=\frac{1}{4},
  \end{eqnarray*}
whence we see that $X_1(3)$ and $X_1(1)$ are equally  
distributed random variables.
 \end{example}

Example 1.5 
addresses  the following curious question.

% Ques 1.7
\begin{question} 
Do there exist positive integers $N\ge 5$, $l$
 and $m$ such that $2\le m\le N-2$ and $1\le l\le N-1$ for 
which at least one of the following two assertions 
there holds:
   \begin{itemize}
\item[(i)] the real part $U_l(m,N)$ of the random variable $X_l(m,N)$
is uniformly distributed{\rm ;}
  \item[(ii)] 
the imaginary part $V_l(m,N)$ of the random variable $X_l(m,N)$
is uniformly distributed{\rm ?}
  \end{itemize}
 \end{question} 

Let us now briefly describe the organization of the paper. 
In Section 2 we give our main results followed by some remarks.
Some of these results are also proved or attributed in \cite{ssa}
and extended in \cite{m7}.
Three examples   and related two assertions concerning certain 
classes of the  random variables $X_l(m)$ are presented 
in  Section 3. As applications, some  combinatorial congruences
are proved. In the last section, we give proofs of the results
of Section 2. 

% Section 2
\section{The main results}

The following 
{\it antisymmetric property} of the random variables $X_l(m,N)$, 
$U_l(m,N)$ and $V_l(m,N)$  should  be 
useful for related computational purposes.

% Prop. 2.1
  \begin{proposition}
Let $N\ge 2$, $l$ and $m$  be  nonnegative integers 
such that $0\le l\le N-1$ and $1\le m\le N-1$. Then the random 
variables $X_l(m,N)$  and $-X_l(N-m,N)$ are equally distributed.
The same assertion holds for the random variables $U_l(m,N)$ and $V_l(m,N)$. 
   \end{proposition}

Since $X_l(1,N)$ is the uniform random variable with
$\mathrm{Prob}\left(X_l(1,N)=e^{-j2il\pi/N}\right)=1/N$
for every $l\in\{0,1,\ldots,N-1\}$, it follows from Proposition 2.1
that   $X_l(N-1,N)$ is also the uniform random variable with
$\mathrm{Prob}\left(X_l(N-1,N)=-e^{-j2il\pi/N}\right)=1/N$
for every $l\in\{0,1,\ldots,N-1\}$.

Let $X_l(m,N)=U_l(m,N)+jV_l(m,N)$ be a  random variable from Definition
1.2, where $U_l(m,N)$ and $V_l(m,N)$ be its real and imaginary part,
respectively.   Since obviously, the set $\Phi(l,N)$ given by (3)
can also be expressed in the form
   $$
\Phi(l,N)=\{\overline{w},\overline{w^2},\ldots,\overline{w^N}\},
  $$
we immediately have the following result.

% Prop. 2.2
    \begin{proposition}
Let $N\ge 2$, $l$ and $m$  be  nonnegative integers 
such that $0\le l\le N-1$ and $1\le m\le N$. Then the
 imaginary part $V_l(m,N)$ of the  random variable $X_l(m,N)$
is symmetrically distributed around zero (i.e., around the mean of 
$V_l(m,N)$) in the sense that
for each value $x$ of $V_l(m,N)$ there holds
  $$
\mathrm{Prob}(V_l(m,N)=-x)=\mathrm{Prob}(V_l(m,N)=x).
  $$ 
     \end{proposition}
As an immediate consequence of Proposition 2.2, we obtain the 
following result.
% Cor. 2.3
 \begin{corollary}
Let $k$ be any  positive odd integer. Then the $k$th moment 
$\mu_k[V_l(m,N)]$ of the random variable $V_l(m,N)$ defined above
is equal to zero, that is, 
   $$
\mu_k[V_l(m,N)]:=\Bbb E[(V_l(m,N))^k]=0.\leqno(17)
   $$
  \end{corollary}
 In the next section we give  a direct combinatorial proof of  
the following expressions (for another proof of (18) and (19) see \cite{ssa}).  

% Theorem 2.4 
 \begin{theorem} 
Let $N\ge 2$, $l$ and $m$ be positive integers such that
 $1\le l\le N-1$ and $1\le m\le N$.  Then the expected value and the variance  
of the random variable $X_l(m,N)$ from Definition $1.2$ are respectively given
 by
  $$
\Bbb E[X_l(m,N)]=0,\leqno(18) 
   $$
and
   $$
{\rm Var}[X_l(m,N)]=\Bbb E[|X_l(m)|^2]=\frac{m(N-m)}{N-1}.\leqno(19)
   $$
If we put $X_l(m,N)= U_l(m,N)+jV_l(m,N)$, where 
$U_l(m,N)$ is the {\it real part} 
 and $V_l(m,N)$ is the {\it imaginary part} of $X_l(m,N)$, then
 $$
\Bbb E[U_l(m,N)]=\Bbb E[V_l(m,N)]=0.\leqno(20) 
   $$

If in addition, we suppose that $1\le l\le N-1$ and  $N\not= 2l$, then
  $$
\Bbb E[(U_l(m,N))^2]=\Bbb E[(V_l(m,N))^2]=\frac{m(N-m)}{2(N-1)}.
\leqno(21)
  $$
 \end{theorem}     

As  consequences of Theorem 2.4, we can easily obtain the following two 
results.

% Cor 2.5
 \begin{corollary}
Let $N\ge 2$, $l$ and $m$ be positive integers such that
 $1\le l\le N-1$ and $1\le m\le N$.
If we take $X_l(m,N)= U_l(m,N)+jV_l(m,N)$, then 
    \begin{equation*}\begin{split}
\Bbb E[(X_l(m,N))^2]&=\Bbb E[(U_l(m,N))^2]+\Bbb E[(V_l(m,N))^2]\\
(22)\qquad\qquad\qquad\qquad\qquad\qquad &={\rm Var}[X_l(m,N)]=
\frac{m(N-m)}{N-1},\qquad\qquad\qquad\qquad\qquad\qquad\quad
   \end{split}\end{equation*}
where $\Bbb E[(X_l(m,N))^2]$ is defined by $(14)$.
 \end{corollary}

% Cor 2.6
 \begin{corollary}
Let $N\ge 2$, $l$ and $m$ be positive integers such that
 $1\le l\le N-1$, $N\not= 2l$ and  $1\le m\le N$.
 If we put $X_l(m,N)= U_l(m,N)+jV_l(m,N)$, then 
    $$
{\rm Var}[U_l(m,N)]={\rm Var}[V_l(m,N)]=
\frac{m(N-m)}{2(N-1)},\leqno(23)
    $$
where ${\rm Var}[U_l(m,N)]$ and ${\rm Var}[V_l(m,N)]$
are the variances of $U_l(m,N)$ and $V_l(m,N)$, respectively.
 \end{corollary}

% Rem 2.7
\begin{remark} Notice that the cases $l=0$ and $N=2l$ 
which are excluded from the above three assertions correspond to
the real-valued cases $X_0(m,N)$ and $X_l(m,2l)$  of the random variable 
$X_l(m,N)$ considered by Examples 3.1 and 3.2, respectively.
   \end{remark}
From the expression (19) we see that the value ${\rm Var}[X_l(m,N)]$
does not depend on $l$. We believe that this fact would be important 
and helpful for some further  investigations of  certain classes of the 
random variables $X_l(m,N)$ and related applications. 

Here we also extend the expression (18) of Theorem 2.4 as follows.

% Theorem 2.8 
\begin{theorem} 
Let $N$, $l$, $m$ and $k$  be positive integers such that
$1\le l\le N-1$ and $1\le m\le N$. If $k$ is not divisible by 
$N/\gcd(N,l)$  $($$\gcd(N,l)$ denotes  the greatest 
common divisor of $N$  and $l$),
then the $k$th moment $\mu_k$ of the  random variable $X_l(m,N)$ from 
Definition $1.2$ is equal to zero, i.e., 
   $$
\mu_k:=\Bbb E[(X_l(m,N))^k]=0.\leqno(24)
   $$
  \end{theorem}

%Remark 2.9
\begin{remark}
Notice that the equality (18) from Theorem 2.4 ia a particular case of the 
equality (24) with $k=1$. However, in Section 4, we give a direct proof of (18). 
  \end{remark}

In view of Theorem 2.8, it remains an open problem to calculate 
$\Bbb E[(X_l(m,N))^k]$ in the case when $k$ is  divisible by 
$N/\gcd(N,l)$. From (16) of Example 1.5 we see that generally,  in this case
$\Bbb E[(X_l(m,N))^k]\not= 0$.
However, we  are able to prove the following itself interesting result. 

% Prop. 2.10
\begin{proposition} 
  Let $N$, $l$, $m$ and $k$  be nonnegative integers such that
$0\le l\le N-1$, $1\le m\le N$ and $k\ge 1$.
Then the $k$th moment 
$\Bbb E[(X_l(m,N))^k]$ of the  random variable 
$X_l(m,N)$ is a real number.
\end{proposition}

Notice that in Section 4 we give a constructive proof of Proposition 2.10 
which is based on Newton's identities (Newton-Girard formula).

%Remark 2.11
\begin{remark} 
Let ${\rm{\bf A}}$ be a $m\times n$ matrix over
the field $\Bbb C$ (or $\Bbb R$)  and let 
${\rm{\bf a_1}},\ldots, {\rm{\bf a_n}}\in \Bbb C^m$ (or $\in \Bbb R^m$) 
be its columns. Then the {\it coherence} 
of ${\rm{\bf A}}$ is the number $\mu({\rm{\bf A}})=\mu$ defined as
     $$
\mu=\max_{1\le i<j\le n}\frac{|\langle {\rm{\bf a_i}}, {\rm{\bf a_j}}
\rangle|}{\Vert {\rm{\bf a_i}}\Vert_2\cdot \Vert 
{\rm{\bf a_j}}\Vert_2}.
       $$
It was noticed in \cite[p. 159]{ss} (also see \cite{s2}) that  
the ratio $\frac{\sigma[X_l(m,N)]}{m}=\sqrt{\frac{N-m}{m(N-1)}}$ 
(where $\sigma[X_l(m,N)]=\sqrt{{\rm Var}[X_l(m,N)]}$
with ${\rm Var}[X_l(m,N)]$ given by (19)) is a crucial parameter 
({\it Welch bound} \cite{we} for coherence 
$\mu$ of  measurement matrix ${\rm{\bf A}}$) for corrected signal detection.
More precisely  (for a particularly 
elegant and very short proof of this bound see \cite{jmf}; also see 
\cite[Chapter 5, Theorem 5.7]{fr}), 
the
coherence $\mu$ of a matrix ${\rm{\bf A}}\in {\Bbb K}^{m\times N}$, where the 
field $\Bbb K$ can either be $\Bbb R$ or $\Bbb C$, with $l_2$-normalized 
columns satisfies the inequality
   $$
\mu\ge \sqrt{\frac{N-m}{m(N-1)}},
  $$  
which under above notation can be written as 
   $$
\mu \ge \frac{\sigma[X_l(m,N)]}{m}.
   $$
Equality in the above two inequalities holds if and only if the columns 
$\rm{\bf{a}}_1,\ldots ,\rm{\bf{a}}_N$ of the matrix ${\rm{\bf  A}}$
form an {\it equiangular tight frame}. Ideally, the coherence $\mu$
of a measurement matrix   ${\rm{\bf A}}$ should be small (see 
\cite[Chapter 5]{fr}).

 Let us observe  that if $m \ll N$, then this bound reduces to 
approximately $\mu({\rm{\bf A}})\ge 1/\sqrt{m}$. There is a lot of possible ways 
to construct matrices with small coherence. Not surprisingly, 
one possible option is to consider random matrices ${\rm{\bf A}}$ with 
each entry generated independently at random 
(cf. \cite[Chapter 11]{pi}). 
  \end{remark}

%Remark 2.12
\begin{remark}
Based on some recent results by R. Vershynin  on 
sub-Gaussian random variables (\cite{ver} and \cite{ver2}), 
  some new results concerning
the random variable $X_l(m,N)$ are obtained in \cite{m7}.
In particular, this our investigation is motivated by the 
fact that  {\it  Restricted Isometry Property} (RIP)
introduced in \cite{ct1} holds 
with high probability for any matrix generated by a sub-Gaussian random 
variable (see \cite{ct2} and \cite{rv}). 
      
Furthermore, in \cite{m1} it was generalized the random variable $X_l(m,N)$. 
It was  also derived the expression for 
related expected value and variance. By using these expressions,
  some  probabilistic aspects of 
compressive sensing are considered in the mentioned paper. In particular, 
motivated by 
the observation given in \cite[p. 159]{ss}, the connection between 
Welch bound on the coherence of a particular $m\times N$  matrix 
${\rm{\bf A}}$ over 
$\Bbb C$ and the variance of the  associated random variable
(defined in a suitable manner) was established in \cite{m1}.
  \end{remark}   

% Sec. 3

\section{Some particular cases of the random variables $X_l(m,N)$}

Here we consider some  particular cases of the  random variable 
$X_l(m,N)=X_l(m)$
from Definition 1.2 with different related values $N$, $m$ and $l$.
We believe that these examples will be of interest in future research 
related to the topics of this paper. Firstly, we consider
the  only two cases when $X_l(m)$ is a real-valued discrete random variable,  
or equivalently, when the multiset $\Phi(l,N)$ defined by (2) 
consists of real numbers.

% Exam. 3.1
\begin{example} For $l=0$ and an arbitrary positive integer $N\ge 1$,
the equation (2) yields
  $$
\Phi(0,N)=\{\underbrace{1,\ldots,1}_N\}.\leqno(25)
  $$ 
Then in view of Definition 1.2,
 for any fixed $m$ with $1\le m\le N$,  $X_0(m)$
   is the constant  random variable with
   $$ 
\mathrm{Prob}\left(X_0(m)=m\right)=1.\leqno(26)
  $$
\end{example}

% Exam. 3.2
\begin{example} 
Let $l$ and $N$ be positive integers such that $l/N=1/2$, i.e., $N=2l$.
Then the equation (2) yields
  $$
\Phi(l,2l)=\{\underbrace{1,\ldots,1}_l, 
\underbrace{-1,\ldots,-1}_l\}.\leqno(27)
  $$ 
Let  $m$ be any positive integer such that  $1\le m\le 2l$.
Notice that for each nonnegative integer $k$ with 
$\max \{0,m-l\}\le k\le \min\{m,l\}$  
 from the multiset $\Phi(l,2l)$ we can choose $k$ $1$'s by 
${l\choose k}$ manners and $(m-k)$ $-1$'s by 
${l\choose m-k}$ manners. Since the sum of 
sums of $k$ $1$'s and sums of  $m-k$ $-1$'s is equal to $2k-m$,
it follows that the distribution of the random variable $X_l(m)$ from 
Definition 1.2 is  given by
  $$
\mathrm{Prob}\left(X_l(m)=2k-m\right)=
\frac{{l\choose k}{l\choose m-k}}{{2l\choose m}}
\,\,\textrm{for\,\, each}\,\, k=\max \{0,m-l\},\ldots,\min\{m,l\}.\leqno(28)
  $$      
In particular, if $m=l$, then  (28) yields
 $$
\mathrm{Prob}\left(X_l(l)=2k-l\right)=
\frac{{l\choose k}^2}{{2l\choose l}},
\quad\textrm{for\,\, each}\quad k=0,1,\ldots,l.\leqno(29)
  $$ 
Notice that the distribution given by (28) 
  implies the following special case of one of the most useful identities 
among binomial 
coefficients,   well known  as the {\it Chu-Vandermonde identity} in 
Combinatorics and  Combinatorial Numbwr Theory (see, e.g., 
\cite{ri}):
   $$
\sum_{k=0}^{\min\{m,l\}}{l\choose k}{l\choose m-k}=
{2l\choose m},
   $$
whose special case for $m=l$ is given as 
   $$
\sum_{k=0}^l{l\choose k}^2={2l\choose l}.
   $$
Furthermore, from (29) it follows that the expected value of
$X_l(m)$ is equal to 
    \begin{equation*}\begin{split}
\Bbb E[X_l(m)]=& \sum_{k= \max \{0,m-l\}}^{\min\{m,l\}}(2k-m)
  \frac{{l\choose k}{l\choose m-k}}{{2l\choose m}}\\
= & \frac{1}{{2l\choose m}}\cdot\sum_{k= \max \{0,m-l\}}^{\min\{m,l\}}
(2k-m){l\choose k}{l\choose m-k}=\,\,({\rm substitution}\,\, k=m-t)\\
=& -\frac{1}{{2l\choose m}}\cdot (2t-m)\sum_{t= \max \{0,m-l\}}^{\min\{m,l\}}
{l\choose m-t}{l\choose t}\\
=&-\Bbb E[X_l(m)],
    \end{split}\end{equation*}
whence it follows that 
       $$
\Bbb E[X_l(m)]=0.\leqno(30)
        $$
Notice that from  (19) of Theorem 2.4 we obtain  that the variance of
$X_l(m)$ is equal to 
    $$
{\rm Var}[X_l(m)]=\frac{m(2l-m)}{2l-1}.\leqno(31)
    $$ 
On the other hand, by using (29) and (30), we find that
    \begin{equation*}\begin{split}
{\rm Var}[X_l(m)]=&\Bbb E[(X_l(m))^2]  - (\Bbb E[X_l(m)])^2\\
(32)\qquad\qquad\qquad\qquad=&\Bbb E[(X_l(m))^2]= \sum_{k= 
\max \{0,m-l\}}^{\min\{m,l\}}(2k-m)^2
  \frac{{l\choose k}{l\choose m-k}}{{2l\choose m}}\qquad\qquad\qquad\qquad\\
=& \frac{1}{{2l\choose m}}\cdot\sum_{k= \max \{0,m-l\}}^{\min\{m,l\}}
(2k-m)^2{l\choose k}{l\choose m-k}.
       \end{split}\end{equation*}
By comparing the equalities (31) and (32), we obtain the following 
combinatorial identity: 
    $$
\sum_{k= \max \{0,m-l\}}^{\min\{m,l\}}
(2k-m)^2{l\choose k}{l\choose m-k}= \frac{m(2l-m)}{(2l-1)}{2l\choose m}
\,\, {\rm with}\,\, 1\le m\le 2l.\leqno(33)
   $$
If we take $l=m$ into (33), then it becomes 
   $$
\sum_{k=0}^m(2k-m)^2{m\choose k}^2=\frac{m^2}{2m-1}{2m\choose m},
   $$
whence by using the identity 
$\frac{m^2}{2m-1}{2m\choose m}=2m{2m-2\choose m-1}$, we get
the following curious combinatorial identity.
\end{example}

% Ident. 3.3
\begin{identity}
Let $m$ be an arbitrary positive integer. Then
   $$
\sum_{k=0}^m(2k-m)^2{m\choose k}^2=2m{2m-2\choose m-1}.
   $$ 
\end{identity}

% Rem. 3.4
\begin{remark} 
Smilarly as in Example 3.2, we can obtain several combinatorial identities. 
For example, 
by determining directly the variance ${\rm Var}[U_l(m,3l)]$ associated to the 
multiset      $\Phi(l,3l)=\{\underbrace{1,\ldots,1}_l,
\underbrace{-1/2,\ldots,-1/2}_{2l}\}$ ($1\le m\le 3l$)  
and  using the expression (19) of Theorem 2.4, we arrive at the 
following identity:
   $$
\sum_{k=\max\{0,m-l\}}^{\min\{m,2l\}}(2m-3k)^2
{l\choose m-k}{2l\choose k}=\frac{2m(3l-m)}{3l-1}{3l\choose m}.
  $$
In particular, for $l=m$ the above congruence becomes 
  $$
\sum_{k=0}^m(2m-3k)^2
{m\choose k}{2m\choose k}=\frac{4m^2}{3m-1}{3m\choose m}.
  $$

Similarly, by determining directly the variance ${\rm Var}[U_l(m,6l)]$ 
associated to the multiset   $\Phi(l,6l)=\{\underbrace{1,\ldots,1}_l,
\underbrace{-1,\ldots,-1}_l, \underbrace{1/2,\ldots,1/2}_{2l},
\underbrace{-1/2,\ldots,-1/2}_{2l}\}$ ($1\le m\le 6l$), we obtain 
the following identity:
  $$
\sum_{\sum_{i=1}^4m_i=m}(2m_1-2m_2+m_3-m_4)^2{l\choose m_1}{l\choose m_2}
{2l\choose m_3}{2l\choose m_4}=\frac{2m(6l-m)}{(6l-1)}{6l\choose m},
  $$
where the summation ranges over all nonnegative integers 
$m_i$  ($i=1,2,3,4$) such that  $\sum_{i=1}^4m_i=m$. 
  \end{remark}

% Exam. 3.5
\begin{example}
Let  $N=p$ be any prime number   and let $m$ be a positive integer
 such that $1\le m\le p$.
Then for $l=1$, consider the set consisting of all $p$th roots of the unity.
Then if we put  $\varepsilon = e^{-2j\pi/p}$, we have
 $$
\Phi(1,p)=\{1,\varepsilon,\varepsilon^2,\ldots,\varepsilon^{p-1}\}.
   $$
Now we will prove that the random variable 
$X_1(m,p)$ from Definition 1.2
is the uniform random variable whose distribution is given by
     $$
\mathrm{Prob}\left(X_1(m,p)= 
\varepsilon^{n_1}+ \varepsilon^{n_2}+\cdots+\varepsilon^{n_m}\right)=
     \frac{1}{{p\choose m}},\leqno(34)
     $$ 
where $\{n_1,n_2,\ldots, n_m\}$ is any subset of $\{0,1,2,\ldots,p-1\}$
 such that $0\le n_1<n_2<\cdots <n_m\le p-1$.
In order to show this fact, for the sake of completeness, 
we will prove the known 
fact in Number Theory that for every prime number $p$ the polynomial
$P_{p-1}(x)$ defined as     
   $$
P_{p-1}(x)=1+x+\cdots+x^{p-1}, \,\, x\in\Bbb R,\leqno(35)
  $$
is an {\it irreducible polynomial} of degree $p-1$ over the field 
$\Bbb Q$ of rational numbers (or equivalently, in the {\it ring} 
$\Bbb Z[x]$ of polynomials with integer coefficients). Namely, since $P_{p-1}(x)=(x^p-1)/(x-1)$ for each $x\not=1$,
then by replacing $x-1=y$, i.e., $x=y+1$, and using the binomisl 
expansion, we find that for each $y\not=0$ there holds
  \begin{equation*}\begin{split}  
P_{p-1}(x)&=P_{p-1}(y+1)=\frac{(y+1)^p-1}{y}\\
&=\frac{\sum_{k=1}^{p}{p\choose k}y^k}{y}=
\sum_{k=1}^{p}{p\choose k}y^{k-1}=y^{p-1}+
\sum_{k=1}^{p-1}{p\choose k}y^{k-1}.  
    \end{split}\end{equation*}
Applying the well known classical  {\it Eisenstein's irreducibility 
criterion} \cite{e} from Number Theory to the above 
expression for the polynomial $P_{p-1}(x)$, and using the fact that 
by {\it Kummer's theorem} (see, e.g., \cite[Section 2, page 6]{m3}),
the binomial coefficient ${p\choose k}$ is divisible by a prime 
$p$ for every $k=1,2,\ldots,p-1$, it follows that $P_{p-1}(x)$
is an irreducible polynomial over the field $\Bbb Q$ of rational numbers.
Hence,  the polynomial  $P_{p-1}(x)$ given by (35) 
is the {\it minimal polynomial} of its root 
$\varepsilon = e^{-2j\pi/p}$ over 
the field $\Bbb Q$ of rational numbers.

 Now if we suppose  that  for some
two distinct subsets $\{n_1,n_2,\ldots, n_m\}$ and 
 
\noindent $\{t_1,t_2,\ldots, t_m\}$
of the  set $\{0,1,2,\ldots,p-1\}$ there holds    
   $$
\varepsilon^{n_1}+ \varepsilon^{n_2}+\cdots+\varepsilon^{n_m}=
\varepsilon^{t_1}+ \varepsilon^{t_2}+\cdots+\varepsilon^{t_m},
   $$
then obviously, the above equality can be reduced to  the form
  $$
\sum_{i=0}^{p-1}\alpha_i\varepsilon^i=0,\leqno(36)
 $$
where the coefficients $\alpha_i\in\{0,-1,1\}$, at least two 
$\alpha_i\not=0$ for some $i\in\{0,1,2,\ldots,p-1\}$
and at least one $\alpha_k=1$ for some $k\in\{1,2,\ldots,p-1\}$.
Therefore, in view of the above  fact that the  polynomial $P_{p-1}(x)$ 
defined by (35) is the minimal polynomial of $\varepsilon$ over 
the field $\Bbb Q$, we conclude that  the expression on the left hand side of 
(36) is $\not=0$. A contradiction, and thus for all two distinct subsets 
$\{n_1,n_2,\ldots, n_m\}$ and $\{t_1,t_2,\ldots, t_m\}$
of the  set $\{0,1,2,\ldots,p-1\}$ there holds    
   $$
\varepsilon^{n_1}+ \varepsilon^{n_2}+\cdots+\varepsilon^{n_m}\not=
\varepsilon^{t_1}+ \varepsilon^{t_2}+\cdots+\varepsilon^{t_m}.
   $$  
This shows that for every prime number $p$, $X_1(m,p)$ is the uniform random 
variable with distribution given  by (34) and its range consists of 
${p\choose m}$ elements.  

 If $l$ is any positive integer such that 
$l\le p-1$, then in view of the fact that $N=p$ is a prime number,
we have 
   $$
\Phi(l,p)=\Phi(1,p)=
\{1,\varepsilon,\varepsilon^2,\ldots,\varepsilon^{p-1}\}.
   $$
This together with the result proved above yields the following assertion.
        \end{example}   

% Claim 3.6
     \begin{claim}
 Let $N=p$ be any prime number  and let $m$ be a positive integer
 such that $1\le m\le p$. Then $X_1(m),X_2(m),\ldots,X_{p-1}(m)$
are equally distributed uniform random variables whose  distribution law
is given by $(34)$.
     \end{claim}

For a given  prime number  $p$ and a nonnegative integer $k$
consider the measurement  row matrix ${\rm{\bf A}}$ (the basis function)  
defined by 
 $$
{\rm{\bf A}}=\left(e^{-2kj\pi/p},e^{-4kj\pi/p},
\ldots ,e^{-2pkj\pi/p}\right).
 $$ 
Let $k_0$  and   $m$
be  nonnegative integers such that $k_0\not= k$ and  $1\le m\le p$
(cf. \cite{ssa}). Let $\Psi_m$
denote  the subset of all vectors (signals) $x\in\Bbb C^p$ whose elements are 
$p$-tuples of the form 
 $$
\left(\delta_1 e^{2k_0j\pi/p},\delta_2 e^{4k_0j\pi/p},
\ldots , \delta_p e^{2pk_0j\pi/p}\right)^T,
 $$ 
where $\delta_s\in \{0,1\}$ for all $s=1,2,\ldots, p$ and 
$\sum_{s=1}^p\delta_s=m$. Notice that the condition 
$\sum_{s=1}^p\delta_s=m$ means  that every (column) vector
$x\in \Psi_m$ has exactly $m$ nonzero coordinates, so that
it is {\it $m$-sparse} vector.
   Then   
from considerations presented in Example 3.5 we immediately get the 
following assertion.
      
% Claim 3.7
     \begin{claim}
Let $y\in \Bbb C$ be a complex number such that under above 
notations and definitions, the equation $Ax=y$ has at least one solution
$x_0\in \Psi_m$. Then this solution is unique in the set 
$\Psi_m$. In other words, in this case the vector $x_0$ is the unique $m$-sparse 
solution of $Ax=y$ with $x_0\in \Psi_m$. 
     \end{claim}

Moreover, the exposition of Example 3.5 obviously yields the following result. 
% Claim 3.8
  \begin{claim}
 Under above notations and definitions, consider the equation 
$Ax=0$, where  $x\in\Psi:= \cup_{m=0}^{p-1}\Psi_m$
($\Psi_0$ denotes the zero vector $(0,0,\ldots,0)\in\Bbb C^p$).
Then the null space $\ker A:=\{x\in \Psi: Ax=0\}$ 
does not contain any  vector $x\in \Psi$ other than the zero 
vector. In other words, the matrix $A$ is injective as a map 
from $\Psi$ to $\Bbb C$. 
  \end{claim}

Of course, the previously proved fact 
that $X_1(m,p)$ is the uniform random variable does not imply
the fact/facts  that its real or/and imaginary part  
  is/are also uniformly distributed (cf. Example 1.4). 

% Rem 3.9
\begin{remark} 
The sufficient condition from Example 3.5 that $N=p$ to be a 
prime number in order that  $X_1(m,p)$ to  be an uniform random variable for 
some (and hence for all) $m$ with $2\le m\le N-1$ 
 is ``probably''  also necessary condition for this assertion. 
This fact is suggested by some heuristic arguments and  the following examples 
of the random variables concerning  the small composite integer values 
of $N$ and $m$:

1)  $N=4$,  $m=2$, $\varepsilon=j$, $X_1(2)=\{1,-1,\varepsilon,-\varepsilon
 \}\Rightarrow 1-1=\varepsilon-\varepsilon=0$; 
   
2) $N=6$, $m=2$; $\varepsilon=(-1+j\sqrt{3})/2$, 
$X_1(2)=\{1,-1,\varepsilon,-\varepsilon, \varepsilon^2,-\varepsilon^2\}
 \} \Rightarrow 1+\varepsilon+\varepsilon^2=-1-\varepsilon-\varepsilon^2=0$;

3) $N=8$, $m=4$; 
$\varepsilon=(1+j)\sqrt{2}/2$, $X_1(4)=\{1,-1,j,-j, \varepsilon,
-\varepsilon, \overline{\varepsilon},-\overline{\varepsilon}\}
 \Rightarrow 1+ (-1)+ j+(-j)=
\varepsilon+(-\varepsilon)+ \overline{\varepsilon}+(-\overline{\varepsilon})
=0$. 
\end{remark}

Some computations and heuristic arguments suggest the following conjecture. 

% CONJ. 3.10
\begin{conjecture} 
Let $N\ge 3$, $l$ and $m$ be positive integers such that
$1\le l\le N$ and $2\le m\le N-1$ and both integers $l$ and $m$ are 
relatively prime to $N$. Then the random variable $X_l(m,N)$
from Definition $1.2$ is uniformly distributed if and 
only if $N$ is a prime number.  
   \end{conjecture}

A Number Theory approach to some probabilistic aspects of compressive
sensing  problems is given in \cite{m5}.

% Sec. 3
\section{Proofs of the results}

In order to prove Theorem 2.4, we will need the following 
known identities.

% Lemma 4.1
\begin{lemma} 
Let $N$ and  $l$ be  positive integers such that $l\le N-1$.
Take $\xi=e^{2jl\pi/N}$. Then 
   $$
\sum_{k=1}^N\xi^k=\sum_{k=1}^N\cos\frac{2kl\pi}{N}
=\sum_{k=1}^N\sin\frac{2kl\pi}{N}=0.\leqno(37)
   $$
If in addition, we suppose that $N\not= 2l$, then
   $$
\sum_{k=1}^N\cos\frac{4kl\pi}{N}=\sum_{k=1}^N\sin\frac{4kl\pi}{N}  =0.  
\leqno(38)
   $$
\end{lemma}
 \begin{proof}
Take 
$$
\xi=\cos\frac{2l\pi}{N}+j\sin\frac{2l\pi}{N}=e^{2jl\pi/N}, S_1=
\sum_{k=1}^N\cos\frac{2kl\pi}{N}\,\, {\rm and}
\,\, S_2=\sum_{k=1}^N\sin\frac{2kl\pi}{N}.
 $$ 
 Then by  de Moivre's formula and the equality $\sum_{t=1}^N\xi^k=0$, 
we immediately obtain 
   \begin{equation*}\begin{split}
S_1+jS_2 &= \sum_{k=1}^N\cos\frac{2kl\pi}{N}+
j\sum_{k=1}^N\sin\frac{2kl\pi}{N}
=\sum_{k=1}^N\left(\cos\frac{2kl\pi}{N}+ j\sin\frac{2kl\pi}{N}\right)\\
&=\sum_{k=1}^N \left(\cos\frac{2kl\pi}{N}+j \sin\frac{2kl\pi}{N}\right)=
\sum_{k=1}^N \left(\cos\frac{2l\pi}{N}+j \sin\frac{2l\pi}{N}\right)^k
\\
&=({\rm since}\,\, \xi \not=1)\quad \sum_{k=1}^N\xi^k=\xi\cdot
\frac{\xi^N-1}{\xi-1}=0.
   \end{split}\end{equation*} 
The above equality shows that $S_1=S_2=0$, which yields (37).

Proceeding in the same manner as above,   with the argument $4kl\pi/N$
instead of $2kl\pi/N$, and using 
the fact that $w:=\cos\frac{4l\pi}{N}+j\sin\frac{4l\pi}{N}\not=1$
(because of the assumption that $N\not= 2l$), 
  we  obtain both identities of  (38).   
    \end{proof}
\begin{proof}[Proof of Proposition $2.1$]

Proof of Proposition 2.1 immediately follows from Definition 1.2 and the 
identities given by (37) of Lemma 4.1.
  \end{proof}

\begin{proof}[Proof of Theorem $2.4$]
For brevity, take $w=e^{-j2l\pi/N}$ and $w_i=w^i$ for every 
$i=1,2,\ldots, N$. Firstly, we consider the case 
when $N$ and  $l$ are  relatively prime  integers. Then  
 the set $\Phi(l,N)$ defined by (3) consists of $N$ distinct elements; 
namely,     
   $$
\Phi(l,N)=\{w_1,w_2,\ldots,w_N\}.\leqno(39)
  $$ 

Then by Definition 1.2, we have
    $$
\Bbb E[X_l(m,N)]=\frac{1}{{N\choose m}}\sum_{\{i_1,i_2,\ldots,i_m\} 
\subset \{1,2,\ldots,N\}}(w_{i_1}+w_{i_2}+\cdots+
w_{i_m}),\leqno(40)
   $$
where the summation ranges over all subsets $\{i_1,i_2,\ldots,i_m\}$
of $\{1,2,\ldots,N\}$ with $1\le i_1<i_2<\cdots <i_s\le N$.
Since any fixed $w_{i_s}$ with $s\in\{1,2,\ldots,N\}$ occurs 
exactly ${N-1\choose m-1}$ times  in the sum on the right hand side 
of (40), and using the fact that $\sum_{i=1}^Nw_i=0$, we find that
     \begin{equation*}\begin{split}
\Bbb E[X_l(m,N)]=&\frac{1}{{N\choose m}}\left({N-1\choose m-1}w_1+
{N-1\choose m-1}w_2+\cdots +{N-1\choose m-1}w_N\right)\\
(41)\qquad\qquad =&
\frac{{N-1\choose m-1}}{{N\choose m}}(w_1+w_2+\cdots +w_N)=0,\qquad\qquad
\qquad\qquad\qquad\qquad\qquad
    \end{split}\end{equation*}
which implies (18). Both equalities from (20) immediately follow 
from (18) in view of the fact that $\Bbb E[X_l(m,N)]=
\Bbb E[U_l(m,N)]+j\Bbb E[V_l(m,N)]$. 
 
If $m=1$, then clearly, $X_l(1,N)$ is the uniform random variable
with $\mathrm{Prob}(X_l(1,N)=w_i)=1/N$ for each $i=1,2,\ldots,N$,
  and so $|X_l(1,N)|$ is the constant random variable
with $\mathrm{Prob}(|X_l(1,N)|=1)=1$.
Then since by (41) $\Bbb E[X_l(m,N)]=0$, we have
   $$
{\rm Var}[X_l(m)]  =\Bbb E[|X_l(m)|^2]-|\Bbb E[X_l(m)]|^2=1.
   $$
The above expression coincides with the expression (19) for 
$m=1$.

Now suppose that $m\ge 2$. Then  we have 
    $$
\Bbb E[|X_l(m)|^2]=\frac{1}{{N\choose m}}\sum_{\{i_1,i_2,\ldots,i_m\} 
\subset \{1,2,\ldots,N\}}(w_{i_1}+w_{i_2}+\cdots+
w_{i_m})(\overline{w_{i_1}+w_{i_2}+\cdots+
w_{i_m}}),\leqno(42)
   $$
where the summation ranges over all subsets $\{i_1,i_2,\ldots,i_m\}$
of $\{1,2,\ldots,N\}$ with $1\le i_1<i_2<\cdots <i_m\le N$.
  Notice that after  multiplication of  terms 
on the right hand side of (42) we obtain that in the obtained sum 
every factor of the form $w_i\bar{w}_i= |w_i|^2$  $(i=1,2,\ldots, N)$ 
occurs exactly 
   ${N-1\choose m-1}$ times, while every factor of the 
form $w_t\bar{w}_s$ with $1\le t<s\le N$, occurs exactly 
   ${N-2\choose m-2}$ times. Accordingly, the equality (42) becomes  
   $$
\Bbb E[|X_l(m)|^2]=\frac{1}{{N\choose m}}\left({N-1\choose m-1}
\sum_{i=1}^N|w_i|^2+
{N-2\choose m-2}\sum_{1\le t<s\le N}w_t\bar{w}_s\right), \leqno(43) 
  $$
whence by using the {\it Pascal's formula} 
${N-1\choose m-1}={N-2\choose m-2}+{N-2\choose m-1}$,
 the facts that $|w_i|=1$ ($i=1,2,\ldots,N$), 
$\sum_{i=1}^N w_i=0$ and the identity ${N\choose m}=\frac{N(N-1)}{m(N-m)}
{N-2\choose m-2}$,  
we obtain
   \begin{equation*}\begin{split}
\Bbb E[|X_l(m)|^2]=&\frac{1}{{N\choose m}}
\left({N-2\choose m-1}\sum_{i=1}^N|w_i|^2\right.\\
  &+\left.\left({N-2\choose m-2}\sum_{i=1}^N|w_i|^2+
{N-2\choose m-2}\sum_{1\le t<s\le N}w_t\bar{w}_s\right)\right)\\
=&\frac{1}{{N\choose m}}
\left({N-2\choose m-1}\sum_{i=1}^N|w_i|^2
+{N-2\choose m-2}\sum_{1\le t\le s\le N}w_t\bar{w}_s\right)\\
=& \frac{1}{{N\choose m}}\left(N{N-2\choose m-1}
+{N-2\choose m-2}\left(\sum_{i=1}^N w_i\right)\left(\sum_{s=1}^N 
\bar{w}_i\right)\right)\\
=&\frac{N{N-2\choose m-1}}{{N\choose m}}=\frac{N{N-2\choose m-1}}
{\frac{N(N-1)}{m(N-m)}{N-2\choose m-1}}=\frac{m(N-m)}{N-1}.
  \end{split}\end{equation*}
From the above expression and (18) we have
  $$
{\rm Var}[X_l(m)]=\Bbb E[|X_l(m)|^2]=\frac{m(N-m)}{N-1}.\leqno(44)
  $$ 
This proves the expression (19). 
 
It remains to prove  the expressions (20) and (21).
Since $w=e^{-j2l\pi/N}$, we have that the real and imaginary part of
 $w^k$ are respectively equal to  $\Re (w^k)= \cos\frac{2kl\pi}{N}$  and 
$\Im (w^k)= -\sin\frac{2kl\pi}{N}$ for every
$k=1,2,\ldots, N$. Then by using the same argument applied in 
the proof of (19) and the assumptions that $1\le l\le N-1$ 
and $N\not= 2l$, we obtain the following analogous equality to (43):
\begin{equation*}\begin{split}
&\Bbb E[(U_l(m))^2]=\frac{1}{{N\choose m}}
\left({N-2\choose m-1}\sum_{k=1}^N\cos^2\frac{2kl\pi}{N}\right.\\
  &+\left.\left({N-2\choose m-2}\sum_{k=1}^N\cos^2\frac{2kl\pi}{N}+
{N-2\choose m-2}\sum_{1\le t<s\le N}\cos\frac{2tl\pi}{N}
\cos\frac{2sl\pi}{N} \right)\right)\\
=&\frac{1}{{N\choose m}}
\left({N-2\choose m-1}\sum_{k=1}^N\cos^2\frac{2kl\pi}{N}
+{N-2\choose m-2}\left(\sum_{t=1}^N \cos\frac{2tl\pi}{N}\right)^2\right)\\
(45)\qquad=& \frac{{N-2\choose m-2}}{{N\choose m}}
\sum_{k=1}^N\cos^2\frac{2kl\pi}{N}
= \frac{{N-2\choose m-2}}{{N\choose m}}\sum_{k=1}^N
\left(\frac{1+ \cos\frac{4kl\pi}{N}}{2}\right)\qquad\qquad\\
=& \frac{{N-2\choose m-2}}{{N\choose m}}
\left(\frac{N}{2}+\sum_{k=1}^N \cos\frac{4kl\pi}{N}\right)=
\frac{{N-2\choose m-2}}{{N\choose m}}\cdot\frac{N}{2}
=\frac{{N-2\choose m-2}}{\frac{N(N-1)}{m(N-m)}{N-2\choose m-2}}
\cdot\frac{N}{2}\\
=&\frac{m(N-m)}{2(N-1)}.
  \end{split}\end{equation*}
 This proves the first equality of (20). Using this equality, (44)
and the equality
   $$
\Bbb E[|X_l(m)|^2]=\Bbb E[(U_l(m))^2]+\Bbb E[(V_l(m))^2],
   $$
we obtain 
    $$
\Bbb E[(V_l(m))^2]=\frac{m(N-m)}{2(N-1)}.
   $$
which together with (45) implies (21). This completes proof of Theorem 2.4.
\end{proof}

\begin{proof}[Proof of Corollary  $2.5$]
In order to prove Corollary 2.5, observe that by (21) of 
Theorem 2.4, we have 
  \begin{equation*}\begin{split}
\Bbb E[(X_l(m))^2] &=\Bbb E[(U_l(m))^2]+
\Bbb E[(V_l(m))^2]-2j\Bbb E[U_l(m)V_l(m)]\\
&=\frac{m(N-m)}{N-1}-2j\Bbb E[U_l(m)V_l(m)].
  \end{split}\end{equation*}
Therefore, the equalities (22) are
 equivalent
to the following one: 
    $$
\Bbb E[U_l(m)V_l(m)]=0.\leqno(46)
    $$
Observe that 
        \begin{equation*}\begin{split}
&\Bbb E[U_l(m)V_l(m)]\\
(47)\qquad  &=\frac{1}{{N\choose m}}
\sum_{\{k_1,k_2,\ldots,k_m\} 
\subset \{1,2,\ldots,n\}\atop \{s_1,s_2,\ldots,s_m\} 
\subset \{1,2,\ldots,n\}}\left( \cos\frac{2k_1l\pi}{N}+
\cos\frac{2k_2l\pi}{N}+\cdots+\cos\frac{2k_ml\pi}{N}\right)\times\qquad  
\\
&\left( \sin\frac{2s_1l\pi}{N}+
\sin\frac{2s_2l\pi}{N}+\cdots+\sin\frac{2s_ml\pi}{N}\right),\\
   \end{split}\end{equation*}
where the summation ranges over all subsets $\{k_1,k_2,\ldots,k_m\}$
and $\{s_1,s_2,\ldots,s_m\}$
of $\{1,2,\ldots,n\}$ with $1\le k_1<k_2<\cdots <k_m\le N$ and
 $1\le s_1<s_2<\cdots <s_m\le N$.
After multiplication of   terms 
on the right hand side of (47) we obtain that in the obtained sum 
 every factor of the form 
$\cos\frac{2kl\pi}{N}\sin\frac{2sl\pi}{N}$
($k,s=1,2,\ldots,N$)   occurs exactly ${N-1\choose m-1}^2$
times in related sum. Therefore, by using the 
trigonometric identity $\cos\alpha\sin\beta= (\sin(\alpha+\beta) 
+\sin(\beta-\alpha))/2$ and  the identity (37) of Lemma 4.1, we have
 \begin{equation*}\begin{split}
&\Bbb E[U_l(m)V_l(m)]=\frac{1}{{N\choose m}^2}
\left({N-1\choose m-1}^2 \sum_{k=1}^N\sum_{s=1}^N\cos\frac{2kl\pi}{N}
\sin\frac{2sl\pi}{N}\right)\\
&=\frac{1}{2{N\choose m}^2}
\left({N-1\choose m-1}^2 \sum_{k=1}^N\sum_{s=1}^N\left(
\sin\frac{2(k+s)l\pi}{N}+\sin\frac{2(s-k)l\pi}{N}
\right)\right)\\
&=\frac{1}{2{N\choose m}^2}\left(\sum_{k=1}^N\sum_{s=1}^N
\sin\frac{2(k+s)l\pi}{N}+\sum_{k=1}^N\sum_{s=1}^N\sin\frac{2(s-k)l\pi}{N}
\right)\\
&({\rm because\,\, of\,\, the\,\, periodicity\,\, of\,\, 
the\,\, function\,\,} \sin x)\\
&=\frac{1}{2{N\choose m}^2}\left(N\cdot\sum_{s=1}^N
\sin\frac{2sl\pi}{N}+N\cdot\sum_{s=1}^N\sin\frac{2sl\pi}{N}
\right)=0.
  \end{split}\end{equation*}
Hence, the equality (46) holds and the proof of the corollary  is completed.
  \end{proof}

\begin{proof}[Proof of Corollary $2.6$]. Both equalities 
given by (23) immediately follow from  the expressions (20)  and (21) of 
Theorem 2.4, taking into account  that   
${\rm Var}[X^2]=\Bbb E[X^2]-(\Bbb E[X])^2$
holds for arbitrary real-valued random variable $X$  with the expected 
value $\Bbb E[X]$ and the variance  ${\rm Var}[X]$.
  \end{proof}

\begin{proof}[Proof of Theorem $2.8$]
  Take $w=e^{-j2l\pi}/N$. Then the multiset  $\Phi(l,N)$ defined by (3) 
can be written as     
  $$
\Phi(l,N)=\{1,w,w^2,\ldots,w^{N-1}\}.  $$ 
Notice that by Definition 1.2, the random variable $X_l(m)$ is 
``uniformly''  defined on the set $\Sigma_m$  of all $m$-element sums of 
$\Phi(l,N)$, 
i.e., on the set  consisting of all sums formed of some $m$ elements of the  
set $\Phi(l,N)$.
Therefore, the random variable $(X_l(m))^k$ is ``uniformly'' defined on the set 
     $$
S_k:=\{(w^{i_1}+w^{i_2}+\cdots+ w^{i_m})^k: 0\le i_1<i_2<\cdots< i_{m}\le N-1
 \}.
       $$ 
Notice that the set $S_k$ is invariant under multiplication by 
$w^k$, i.e., there holds
   $$
w^kS_k:=\{w^kz:\, z\in S_k\}=S_k.
   $$
Accordingly, and taking into account that the random variable $X_l(m)$ is 
``uniformly''  defined on the set $\Sigma_m$ in the sense that 
$\mathrm{Prob}\left(X_l(m)= z\right)=1/{{N\choose m}}$ for each
$z\in \Sigma_m$, we conclude that the random variables $(X_l(m))^k$
and $w^k(X_l(m))^k$ have the same distribution. Therefore, we have 
 $$
\Bbb E [(X_l(m))^k]=\Bbb E [w^k(X_l(m))^k],
 $$
whence taking $\Bbb E [w^k(X_l(m))^k]=w^k \Bbb E [(X_l(m))^k]$, we obtain 
 $$
(1-w^k)\Bbb E [(X_l(m))^k]=0.
 $$
Since by the  assumption of the theorem,  $k$ is not divisible by $N/\gcd(N,l)$,
it follows that $kl/N$ is not an integer and thus, 
$w^k=\cos \frac{2kl\pi}{N}-j\sin \frac{2kl\pi}{N}\not= 1$.
In view of this fact, the above equality yields 
 $$
\Bbb E [(X_l(m))^k]=0,
  $$
as desired.
 \end{proof}

%Proof of Proposition 2.10 
 \begin{proof}[Proof of Proposition $2.10$]
For brevity, take $w=e^{-2jl\pi/N}$. First notice that the assertion 
holds for $l=0$ 
since $X_0(m)$  is the constant  random variable such that
   $$
\mathrm{Prob}\left(X_0(m)=m\right)=1.
   $$ 

Now suppose that $1\le l\le N-1$.
By definition of $\Bbb E[\left(X_l(m)\right)^k]$ and using the additive 
property for the expectation, we find that
   $$
\Bbb E[\left(X_l(m)\right)^k]=
\frac{1}{{N\choose m}}\sum_{\{i_1,i_2,\ldots,i_m\} 
\subset \{1,2,\ldots,N\}}(w^{i_1}+w^{i_2}+\cdots+
w^{i_m})^k,\leqno(48)
   $$
where the summation ranges over  all subsets $\{i_1,i_2,\ldots,i_m\}$
of $\{1,2,\ldots,N\}$ with $1\le i_1<i_2<\cdots <i_m\le N$.
Consider the polynomial $P_k(x_1,\ldots,x_N)\in 
\Bbb R[x_1,\ldots,x_N]$ of $N$ real variables 
$x_1,\ldots,x_N$ defined as 
  $$
P_k(x_1,\ldots,x_N)=\frac{1}{{N\choose m}}\sum_{\{i_1,i_2,\ldots,i_m\} 
\subset \{1,2,\ldots,N\}}(x_{i_1}+x_{i_2}+\cdots+x_{i_m})^k,\leqno(49)
   $$
where the summation ranges over all subsets $\{i_1,i_2,\ldots,i_m\}$
of $\{1,2,\ldots,N\}$ with $1\le i_1<i_2<\cdots <i_m\le N$.
Clearly, $P_k$ is a homogeneous symmetric polynomial of degree
$k$. Let us recall that a polynomial in $n$ real (or complex) variables,
$P(x_1,\ldots,x_n)\in \Bbb R[x_1,\ldots,x_n]$
(or $P(x_1,\ldots,x_n)\in \Bbb C[x_1,\ldots,x_n]$) is known 
as a {\it symmetric polynomial} if for any permutation 
$\sigma$ of the set $\{1,2,\ldots ,n\}$, 
$P(x_{\sigma(1)},\ldots,  x_{\sigma(n)})=P(x_1,\ldots,x_n)$.
  
Then by {\it fundamental theorem of symmetric functions},
the polynomial $P_k$ defined by (49)  can be expressed as a polynomial in the 
{\it elementary symmetric polynomials} on the variables $x_1,\ldots,x_N$, 
i.e.,
   $$
 P_k(x_1,\ldots,x_N)=Q_k(\sigma_1,\ldots,\sigma_N),\leqno(50)
   $$
where $Q_k$ is a polynomial in 
$\Bbb R[x_1,\ldots,x_N]$ and $\sigma_s$ ($s=1,\ldots, N$) are elementary 
symmetric polynomials in $\Bbb R[x_1,\ldots,x_N]$  defined as 
  $$
\sigma_s(x_1,\ldots,x_N)=
\sum_{\{i_1,i_2,\ldots,i_s\} 
\subset \{1,2,\ldots,N\}}x_{i_1}\cdots x_{i_s},  
  $$
where the summation ranges over all subsets $\{i_1,\ldots,i_s\}$
of $\{1,\ldots,N\}$ with $1\le i_1<\cdots <i_s\le N$.
The  $N$th {\it power sum} (or the $N$th {\it power symmetric 
function}) $p_n\in\Bbb R[x_1,\ldots,x_N]$ is defined as
   $$
p_n(x_1,\ldots,x_N)=\sum_{i=1}^Nx_i^n.
  $$
Then by {\it Newton's identities} (also known as the 
{\it Newton-Girard formula}; see, e.g., \cite{me}; 
cf. \cite[Lemma 2.1]{m6}), for all $n\ge 1$ there 
holds
   $$
p_n(x_1,\ldots,x_N)=(-1)^{n-1}n\sigma_n(x_1,\ldots,x_N)
+ \sum_{i=1}^{n-1}\sigma_{n-i}(x_1,\ldots,x_N)p_i(x_1,\ldots,x_N).\leqno(51)
  $$ 

Let us recall  that the formulae (49), (50) and (51) are also valid 
for the complex values $x_k=w^k$ ($k=1,\ldots,N$).
Accordingly, for all $n\in\Bbb N$ we have 
   $$
p_n(w,\ldots,w^N)= \sum_{k=1}^N w^{kn}=\left\{
 \begin{array}{ll}
0 &\,\, {\rm if}\,\, n\,\,  {\rm is\,\, not\,\, divisible\,\,
 by}\,\, \frac{N}{\gcd(N,l)}\\
 N &\,\, {\rm if}\,\, n\,\,  {\rm is\,\,  divisible\,\,
 by}\,\, \frac{N}{\gcd(N,l)}.\\
  \end{array}\right.\leqno(52)
   $$
We will prove by induction on $n\ge 1$ that $\sigma_n(w,\ldots,w^N)$
 is a real number for all $n\in\Bbb N$. For $n=1$ we have 
    $$
\sigma_1(w,\ldots,w^N)=\sum_{t=1}^Nw^t=0,
   $$
and thus, the induction base holds. Suppose that $\sigma_i(w,\ldots,w^N)$
is a real number for each $i\ge 1$ less than $n$. 
  Then by the identity (51), we have   
  \begin{equation*}\begin{split}
&\sigma_n(x_1,\ldots,x_N)\\
&=\frac{(-1)^{n-1}}{n}\left(p_n(x_1,\ldots,x_N)
- \sum_{i=1}^{n-1}\sigma_{n-i}(x_1,\ldots,x_N)p_i(x_1,\ldots,x_N)
\right).
  \end{split}\end{equation*}
The above formula with $(w,\ldots,w^N)$ instead of $(x_1,\ldots, x_N)$ 
 together with the equalities (52) and the 
induction hypothesis implies that $\sigma_n(w,\ldots,w^N)$ is a real number, 
which finishes the induction proof.
Hence, if we substitute $x_k=w^k$ ($k=1,2,\ldots, N$) into (49) and (50) and  
$a_k=\sigma_k(w,\ldots,w^N)$
  into (50) ($k=1,\ldots,N$), and comparing then (48) and (49), 
we immediately obtain 
   $$
\Bbb E [(X_l(m))^k]=Q_k(a_1,\ldots, a_N).
   $$ 
Since $a_1,\ldots,a_N$  and the all coefficients 
of the polynomial $Q_k$ are  real numbers, we conclude that 
$Q_k(a_1,\ldots, a_N)$ is also a real number. 
Therefore, from the above equality it follows  that 
$\Bbb E [(X_l(m))^k]$ is  a real number. 
This completes  proof of the proposition.
   \end{proof}

\end{document}